\documentclass{amsart}%
\usepackage{graphicx}
\usepackage{amscd, color}
\usepackage{amsmath}
\usepackage{amsfonts}
\usepackage{amssymb}
\usepackage{geometry,amssymb}%
\usepackage[initials]{amsrefs}%
\newtheorem{theorem}{Theorem}[section]

\newtheorem{remark}[theorem]{Remark}

\numberwithin{equation}{section}

\begin{document}
\title[The asymptotic behavior of the constants in the Bohnenblust--Hille inequality]{Estimates for the asymptotic behavior of the constants in the Bohnenblust--Hille inequality}
\author[D. Pellegrino \and G. A. Mu\~{n}oz-Fern\'{a}ndez \and J. B. Seoane-Sep\'{u}lveda]{G. A. Mu\~{n}oz-Fern\'{a}ndez\textsuperscript{*} \and D. Pellegrino\textsuperscript{**} \and J. B. Seoane-Sep\'{u}lveda\textsuperscript{*}}
\address{Departamento de An\'{a}lisis Matem\'{a}tico,\newline\indent Facultad de Ciencias Matem\'{a}ticas, \newline\indent Plaza de Ciencias 3, \newline\indent Universidad Complutense de Madrid,\newline\indent Madrid, 28040, Spain.}
\email{gustavo$\_$fernandez@mat.ucm.es}
\address{Departamento de Matem\'{a}tica, \newline\indent Universidade Federal da Para\'{\i}ba, \newline\indent 58.051-900 - Jo\~{a}o Pessoa, Brazil.} \email{pellegrino@pq.cnpq.br}
\address{Departamento de An\'{a}lisis Matem\'{a}tico,\newline\indent Facultad de Ciencias Matem\'{a}ticas, \newline\indent Plaza de Ciencias 3, \newline\indent Universidad Complutense de Madrid,\newline\indent Madrid, 28040, Spain.}
\email{jseoane@mat.ucm.es}
\thanks{\textsuperscript{*}Supported by the Spanish Ministry of Science and
Innovation, grant MTM2009-07848.}
\thanks{\textsuperscript{**}Supported by CNPq Grant 620108/2008-8, Edital Casadinho.}

\begin{abstract}
A classical inequality due to H.F. Bohnenblust and E. Hille states that for every positive integer $n$ there is a constant $C_{n}>0$ so that
$$\left(\sum\limits_{i_{1},\dots,i_{n}=1}^{N}\left\vert U(e_{i_{^{1}}}, \dots ,e_{i_{n}})\right\vert^{\frac{2n}{n+1}}\right)^{\frac{n+1}{2n}}\leq C_{n}\left\Vert U\right\Vert$$
for every positive integer $N$ and every $n$-linear mapping $U:\ell_{\infty}^{N}\times\cdots\times\ell_{\infty}^{N}\rightarrow\mathbb{C}$. The original estimates for those constants from Bohnenblust and Hille are $$C_{n}=n^{\frac{n+1}{2n}}2^{\frac{n-1}{2}}.$$ In this note we present explicit formulae for quite better constants, and calculate the asymptotic behavior of these estimates, completing recent results of the second and third authors. For example, we show that,  if $C_{\mathbb{R},n}$ and $C_{\mathbb{C},n}$ denote (respectively) these estimates for the real and complex Bohnenblust--Hille inequality then, for every even positive integer $n$,
$$\frac{C_{\mathbb{R},n}}{\sqrt{\pi}} = \frac{C_{\mathbb{C},n}}{\sqrt{2}}  = 2^{\frac{n+2}{8}}\cdot r_n$$
for a certain sequence $\{r_n\}$ which we estimate numerically to belong to the interval $(1,3/2)$ (the case $n$ odd is similar). Simultaneously, assuming that $\{r_n\}$ is in fact convergent, we also conclude that
$$\displaystyle \lim_{n \rightarrow \infty} \frac{C_{\mathbb{R},n}}{C_{\mathbb{R},n-1}} = \displaystyle \lim_{n \rightarrow \infty} \frac{C_{\mathbb{C},n}}{C_{\mathbb{C},n-1}}= 2^{\frac{1}{8}}.$$
\end{abstract}

\maketitle

\section{Preliminaries and background}

Since Lindenstrauss and Pe\l czy\'{n}ski's classical paper \cite{LP}, the theory of
absolutely summing linear operators became an important topic of research
in Functional Analysis (see \cite{DJT} and references therein). The most
famous constant involved in the theory of absolutely linear operators is the
constant from Grothendieck's fundamental theorem in the metric theory of
tensor products, called Grothendieck's constant $K_{G}$. In recent years,
Grothendieck's type inequalities have received a significant amount of
attention in view of their various applications (see, e. g., \cite{naor,
fish}). Grothendieck's famous Resum\'{e} asks for determining the
precise value of $K_{G}$ (see \cite[Problem 3]{res}, \cite{JJJ} and
references therein). However, this problem remains open despite of the progress
made. For instance, it is well-known that for real scalars
$$K_{G}\leq\frac{\pi}{2\log\left(  1+\sqrt{2}\right)}.$$
For some time it was believed that, in fact, this inequality was sharp, but not until very recently \cite{naor2} it was proved that it is actually not.

In the multilinear theory of absolutely summing operators the key constants
are the constants $C_{n}$ involved in the Bohnenblust--Hille inequality, which
we describe below. In 1930 J.E. Littlewood proved that
\[
\left(  \sum\limits_{i,j=1}^{N}\left\vert U(e_{i},e_{j})\right\vert ^{\frac
{4}{3}}\right)  ^{\frac{3}{4}}\leq\sqrt{2}\left\Vert U\right\Vert
\]
for every bilinear form $U:\ell_{\infty}^{N}\times\ell_{\infty}^{N}%
\rightarrow\mathbb{C}$ and every positive integer $N.$ This is the well-known
Littlewood's $4/3$ inequality \cite{Litt}. Just one year later, H.F.
Bohnenblust and E. Hille observed that Littlewood's inequality had important
connection with Bohr's absolute convergence problem which consists in
determining the maximal width $T$ of the vertical strip in which a Dirichlet
series $%
{\textstyle\sum\limits_{n=1}^{\infty}}
a_{n}n^{-s}$ converges uniformly but not absolutely. Bohnenblust and Hille
proved that $T=1/2$ and for this task they improved Littlewood's $4/3$ by
showing that for all positive integer $n>2$ there is a $C_{n}>0$ so that
\begin{equation}
\left(  \sum\limits_{i_{1},...,i_{n}=1}^{N}\left\vert U(e_{i_{^{1}}%
},...,e_{i_{n}})\right\vert ^{\frac{2n}{n+1}}\right) ^{\frac{n+1}{2n}}\leq
C_{n}\left\Vert U\right\Vert \label{hypp}%
\end{equation}
for all $n$-linear mapping $U:\ell_{\infty}^{N}\times\cdots\times\ell_{\infty
}^{N}\rightarrow\mathbb{C}$ and every positive integer $N.$ In their paper
Bohnenblust and Hille showed that $C_{n}=n^{\frac{n+1}{2n}}2^{\frac{n-1}{2}}$
is a valid constant (for related recent papers we refer the reader to
\cite{ann,defant2}).

This inequality was overlooked for a long time and rediscovered later by A.
Davie \cite{Davie} and S. Kaijser \cite{Ka} and the value of $C_{n}$ was
improved to $C_{n}=2^{\frac{n-1}{2}}.$ Also, H. Qu\'{e}ffelec \cite{Que}, A.
Defant and P. Sevilla-Peris \cite{defant2} observed that $C_{n}=\left(
\frac{2}{\sqrt{\pi}}\right) ^{n-1}$ also works in (\ref{hypp}).

In the recent years considerable effort related to the Bohnenblust--Hille
inequality has been made (see \cite{botpams, ann, defant2, defant} and
references therein) but, as it happens to Grothendieck's constant, there are
still open questions regarding the precise value and behavior of the
Bohnenblust--Hille constants. The questions related to the Bohnenblust--Hille
constants, although up to now less investigated than Grothendieck's constant,
seem at least as challenging as those related to Grothendieck's constant.
Besides the estimation of the precise values of $C_{n}$, their asymptotic
growth is also an open problem. In this note we shall be focusing on the
asymptotic behavior of these constants.

In 2010, A. Defant, D. Popa and U. Schwarting \cite{defant} presented a new
proof of the Bohnenblust--Hille Theorem (also valid for the real case) and by
exploring this proof and estimates from \cite{haag} the second and third
authors of this note obtained better estimates for $C_{n}$ \cite{Preprint_PS}.
However these estimates, for big values of $n$, were obtained recursively and
a closed (non recursive) formula could not be easily obtained as well as the
asymptotic growth of the constants.

In this short note we complement the results from \cite{Preprint_PS} and
provide a closed formula (non recursive) for these better estimates in the
Bohnenblust--Hille inequality. We also determine the asymptotic behavior of
these estimates, showing that if $C_{\mathbb{R},n}$ and $C_{\mathbb{C},n}$
denote (respectively) these constants for the real and complex
Bohnenblust--Hille inequality then:

\begin{enumerate}
\item For every even positive integer $n$,
\begin{equation}
\displaystyle C_{\mathbb{R},n}=\left(  \frac{\sqrt{\pi}}{\sqrt{2}}\right)
C_{\mathbb{C},n}=2^{\frac{n+2}{8}}\cdot r_{n}, \label{pb}%
\end{equation}
for certain sequence $\{r_{n}\}_{n}$, which we estimate numerically to belong
to the interval $(1,3/2)$.

\item If $\mathbb{K}=\mathbb{R}$ or $\mathbb{C}$, then
\[
\lim_{n\rightarrow\infty}\frac{C_{\mathbb{K},n}}{C_{\mathbb{K},n-1}}%
=2^{\frac{1}{8}}.\lim_{n\rightarrow\infty}\frac{r_{n}}{r_{n-1}}.
\]
In particular, if $\{r_{n}\}_{n\in\mathbb{N}}$ is in fact convergent (as our
numerical estimates indicate), then%
\[
\lim_{n\rightarrow\infty}\frac{C_{\mathbb{K},n}}{C_{\mathbb{K},n-1}}%
=2^{\frac{1}{8}}.
\]

\item It worths to be mentioned that for complex scalars (\ref{pb}) can be
replaced by smaller constants:%
\[
C_{\mathbb{C},n}=\frac{2^{\frac{n}{8}+\frac{67}{n}+\frac{3}{4}}}{\pi
^{\frac{36}{n}+\frac{1}{2}}}\cdot\left(  3.9296\times10^{-3}\right)
^{1/n}r_{n}%
\]

\end{enumerate}

\section{Bohnenblust--Hille constants: The real case}

The following result appears in \cite{Preprint_PS}, as a consequence of
results from \cite{defant}:

\begin{theorem}
\label{2_2} For every positive integer $n$ and real Banach spaces
$X_{1},\ldots,X_{n},$
\[
\Pi_{(\frac{2n}{n+1};1)}(X_{1},\ldots,X_{n};\mathbb{R})=\mathcal{L}%
(X_{1},\ldots,X_{n};\mathbb{R})\text{ and }\left\Vert .\right\Vert _{\pi
(\frac{2n}{n+1};1)}\leq C_{\mathbb{R},n}\left\Vert .\right\Vert
\]
with
\[
\label{4estrella0}C_{\mathbb{R},2}=2^{\frac{1}{2}}\text{ and }C_{\mathbb{R}%
,3}=2^{\frac{5}{6}},\newline%
\]
\begin{equation}
\label{4estrella}C_{\mathbb{R},n}=2^{\frac{1}{2}}\left(  \frac{C_{\mathbb{R}%
,n-2}}{A_{\frac{2n-4}{n-1}}^{2}}\right) ^{\frac{n-2}{n}}\text{ for }n>3.
\end{equation}
In particular, if $2\leq n\leq14$,%

\begin{equation}
\label{unaestrella1}C_{\mathbb{R},n}=2^{\frac{n^{2}+6n-8}{8n}}\text{ if
}n\text{ is even, and}%
\end{equation}
\begin{equation}
\label{unaestrella2}C_{\mathbb{R},n}=2^{\frac{n^{2}+6n-7}{8n}}\text{ if
}n\text{ is odd.}%
\end{equation}

\end{theorem}

\noindent The above result and the next remark will be crucial for the main
results in this note.

\begin{remark}
\label{rn} Throughout this paper the sequence $\{r_{n}\}_{n\in2\mathbb{N}}$
given by
\[
\displaystyle r_{n}=\frac{1}{2^{\frac{n-2}{4}}\cdot\left[  \prod_{k=1}%
^{\frac{n-2}{2}}\left(  \frac{\Gamma(\frac{6k+1}{4k+2})}{\sqrt{\pi}}\right)
^{2k+1}\right]  ^{1/n}}=\frac{\pi^{\frac{n^{2}-4}{8n}}}{2^{\frac{n-2}{4}}%
\cdot\left[  \prod_{k=1}^{\frac{n-2}{2}}\left(  \Gamma(\frac{6k+1}%
{4k+2})\right)  ^{2k+1}\right]  ^{1/n}}%
\]
shall appear very often. Although it is not proved here, we have strong
numerical evidence supporting the fact that the above sequence is convergent
and, moreover,
\[
\displaystyle r_{n}\thickapprox1.44.
\]
The interested reader can see here below a table with some of the values of
$r_{n}$, for $n$ even.%

\bigskip

\begin{center}
\begin{tabular}
[c]{r|c}%
$n$ & $r_{n}$\\\hline
10 & 1.28682\\
30 & 1.37516\\
50 & 1.39747\\
100 & 1.41640\\
250 & 1.42943\\
500 & 1.43437\\
1,000 & 1.43707\\
5,000 & 1.43951\\
10,000 & 1.43986\\
15,000 & 1.43998\\
25,000 & 1.44008\\
40,000 & 1.44014\\
100,000 & 1.44021\\
300,000 & 1.44023\\
1,000,000 & 1.44025
\end{tabular}
\end{center}
\end{remark}
\bigskip
As we have mentioned before, Theorem \ref{2_2} does not furnish a closed
formula for the constants of Bohnenblust--Hille inequality for big values of
$n$. The estimates are recursive and makes difficult a complete comprehension
of their growth. Our first result is a closed formula for $C_{\mathbb{R},n}$
with $n$ even:

\begin{theorem}
\label{realrn} If $n$ is an even positive integer, then
\[
C_{\mathbb{R},n}=2^{\frac{n+2}{8}}r_{n}.
\]

\end{theorem}

\begin{proof}
Let us begin by noticing that
\[%
\begin{array}
[c]{lcl}%
C_{\mathbb{R},4} & = & \sqrt{2}\left(  \frac{C_{\mathbb{R},2}}{A_{\frac{4}{3}%
}^{2}}\right) ^{\frac{2}{4}},\\
C_{\mathbb{R},6} & = & \sqrt{2}\left[  \frac{\sqrt{2}\left(  \frac
{C_{\mathbb{R},2}}{A_{\frac{4}{3}}^{2}}\right)  ^{\frac{2}{4}}}{A_{\frac{8}%
{5}}^{2}}\right]  ^{\frac{4}{6}}=\frac{\left(  2^{\frac{1}{2}+\frac{1}%
{2}.\frac{4}{6}}\right)  \left(  C_{\mathbb{R},2}\right)  ^{\frac{2}{4}%
.\frac{4}{6}}}{\left(  A_{\frac{4}{3}}^{2}\right)  ^{\frac{2}{4}.\frac{4}{6}%
}\left( A_{\frac{8}{5}}^{2}\right) ^{\frac{4}{6}}},\\
C_{\mathbb{R},8} & = & \sqrt{2}\left[  \frac{\frac{\left( 2^{\frac{1}{2}%
+\frac{1}{2}\cdot\frac{4}{6}}\right)  \left(  C_{\mathbb{R},2}\right)
^{\frac{2}{4}\cdot\frac{4}{6}}}{\left(  A_{\frac{4}{3}}^{2}\right) ^{\frac
{2}{4} \cdot\frac{4}{6}}\left( A_{\frac{8}{5}}^{2}\right)  ^{\frac{4}{6}}}%
}{A_{\frac{12}{7}}^{2}}\right] ^{\frac{6}{8}}=\frac{\left(  2^{\frac{1}%
{2}+\left( \frac{1}{2}+\frac{1}{2}\cdot\frac{4}{6}\right)  \cdot\frac{6}{8}%
}\right)  \left(  C_{\mathbb{R},2}\right) ^{\frac{2}{4} \cdot\frac{4}{6}%
\cdot\frac{6}{8}}}{\left(  A_{\frac{4}{3}}^{2}\right) ^{\frac{2}{4}\cdot
\frac{4}{6}\cdot\frac{6}{8}}\left( A_{\frac{8}{5}}^{2}\right)  ^{\frac{4}%
{6}\cdot\frac{6}{8}}\left(  A_{\frac{12}{7}}^{2}\right) ^{\frac{6}{8}}},
\end{array}
\]
and so on.

Now, using the fact that
\[
A_{p}=\sqrt{2}\left(  \frac{\Gamma((p+1)/2)}{\sqrt{\pi}}\right)  ^{1/p},
\]
we can define
\[%
\begin{array}
[c]{rcl}%
\displaystyle s_{n} & = & \left(  A_{\frac{4}{3}}^{2}\right)  ^{\frac{2}%
{4}\cdot\frac{4}{6}\dots\frac{n-2}{n}}\left(  A_{\frac{8}{5}}^{2}\right)
^{\frac{4}{6}\cdot\frac{6}{8}\dots\frac{n-2}{n}}\left(  A_{\frac{12}{7}}%
^{2}\right)  ^{\frac{6}{8}.\frac{8}{10}\dots\frac{n-2}{n}}\cdot\dots
\cdot\left(  A_{\frac{2n-4}{n-1}}^{2}\right)  ^{\frac{n-2}{n}}=\\
& = & \displaystyle2^{\frac{n-2}{4}}\cdot\left[  \prod_{k=1}^{\frac{n-2}{2}%
}\left(  \frac{\Gamma(\frac{6k+1}{4k+2})}{\sqrt{\pi}}\right)  ^{2k+1}\right]
^{1/n}=\frac{2^{\frac{n-2}{4}}\cdot\left[  \prod_{k=1}^{\frac{n-2}{2}}\left(
\Gamma(\frac{6k+1}{4k+2})\right)  ^{2k+1}\right]  ^{1/n}}{\pi^{\frac{n^{2}%
-4}{8n}}}%
\end{array}
\]
Notice that (see Remark \ref{rn}) $r_{n}=1/s_{n}$. It can be easily checked
that
\begin{equation}
\displaystyle C_{\mathbb{R},n}=2^{\frac{n+2}{8}}\cdot r_{n}.
\end{equation}
Indeed, it suffices with noticing that $C_{\mathbb{R},2}=\sqrt{2}$ and that
\[%
\begin{array}
[c]{lcl}%
\displaystyle C_{\mathbb{R},n}/r_{n} & = & \displaystyle2^{\frac{1}{2}%
+\frac{1}{2}.\frac{\left(  n-2\right)  }{n}+\frac{1}{2}.\frac{\left(
n-4\right)  }{n}+\frac{1}{2}.\frac{\left(  n-6\right)  }{n}+\cdots+\frac{1}%
{2}.\frac{\left(  n-(n-4)\right)  }{n}}.\sqrt{2}^{\frac{2}{n}}\\
& = & \displaystyle2^{\left(  \frac{1}{2}+\frac{1}{2n}\left[  \left(
n-2\right)  +\cdots+4\right]  \right)  +\frac{1}{n}}\\
& = & \displaystyle2^{\frac{n+2}{8}}.
\end{array}
\]

\end{proof}

\begin{remark}
Let us note that, although the previous theorem was proved for $n$ even, a
similar (but not identical) result holds for all $n \in\mathbb{N}$. On the other hand, the
asymptotic behavior of the constants for $n$ odd remains identical. The same
also holds for Theorem \ref{complexrn}.
\end{remark}

Some values of the sequence $\{C_{\mathbb{R},n}\}_{n}$ (using the above
formula for $C_{\mathbb{R},n}$) are shown in the following table.

\bigskip

\begin{center}%
\begin{tabular}
[c]{r|l}%
$n$ & $C_{\mathbb{R},n}$\\\hline
30 & $\thickapprox22 $\\
50 & $\thickapprox126 $\\
100 & $\thickapprox9757$\\
500 & $\thickapprox10^{19}$\\
1,000 & $\thickapprox10^{37}$\\
5,000 & $\thickapprox10^{188}$\\
&
\end{tabular}

\end{center}


The following result will give us the behavior of $\frac{C_{\mathbb{R},n}%
}{C_{\mathbb{R},n-1}}$.

\begin{theorem}
\label{asymptotic_real} For the real case,
\[
\displaystyle\lim_{n\rightarrow\infty}\frac{C_{\mathbb{R},n}}{C_{\mathbb{R}%
,n-2}}=2^{\frac{1}{4}}.\lim_{n\rightarrow\infty}\frac{r_{n}}{r_{n-2}}.
\]
In particular, assuming the convergence of $\{r_{n}\}$,
\[
\displaystyle\lim_{n\rightarrow\infty}\frac{C_{\mathbb{R},n}}{C_{\mathbb{R}%
,n-1}}=2^{\frac{1}{8}}.
\]

\end{theorem}

\begin{proof}
We can work for even and odd integers since the asymptotic behavior for both
cases is the same. The first estimate is clear, since%
$$\frac{C_{\mathbb{R},n}}{C_{\mathbb{R},n-2}}=\frac{2^{\frac{n+2}{8}}r_{n}%
}{2^{\frac{n}{8}}r_{n-2}}=2^{\frac{1}{4}}\cdot\frac{r_{n}}{r_{n-2}}.$$
If $\{r_{n}\}$ is convergent we have
$$2^{\frac{1}{4}}=\lim_{n\rightarrow\infty}\frac{C_{\mathbb{R},n}}%
{C_{\mathbb{R},n-2}}=\lim_{n\rightarrow\infty}\left(  \frac{C_{\mathbb{R},n}%
}{C_{\mathbb{R},n-1}}.\frac{C_{\mathbb{R},n-1}}{C_{\mathbb{R},n-2}}\right)
=\lim_{n\rightarrow\infty}\left(  \frac{C_{\mathbb{R},n}}{C_{\mathbb{R},n-1}%
}\right)^{2}$$
and thus,
$$\lim_{n\rightarrow\infty}\frac{C_{\mathbb{R},n}}{C_{\mathbb{R},n-1}}%
=2^{\frac{1}{8}}.$$
\end{proof}


\section{Bohnenblust--Hille constants: The complex case}

The version of Theorem 2.1 for complex scalars is:

\begin{theorem}
For every positive integer $m$ and complex Banach spaces $X_{1}, \dots, X_{m}%
$,
\[
\Pi_{(\frac{2m}{m+1};1)}(X_{1},...,X_{m};\mathbb{C})=\mathcal{L}%
(X_{1},...,X_{m};\mathbb{C})\text{ and }\left\Vert .\right\Vert _{\pi
(\frac{2m}{m+1};1)}\leq C_{\mathbb{C},m}\left\Vert .\right\Vert
\]
with
\begin{align*}
C_{\mathbb{C},m}  &  =\left(  \frac{2}{\sqrt{\pi}}\right)  ^{m-1}\text{ for
}m=2,3,\\
C_{\mathbb{C},m}  &  \leq\frac{2^{\frac{m+2}{2m}}}{\pi^{1/m}}\left(  \frac
{1}{A_{\frac{2m-4}{m-1}}^{2}}\right)  ^{\frac{m-2}{m}}\left(  C_{\mathbb{C}%
,m-2}\right)  ^{\frac{m-2}{m}}\text{ for }m\geq4.
\end{align*}
In particular, if $4\leq m\leq14$ we have%
\[
C_{\mathbb{C},m}\leq\left(  \frac{1}{\pi^{1/m}}\right)  2^{\frac{m+4}{2m}%
}\left(  C_{\mathbb{C},m-2}\right)  ^{\frac{m-2}{m}}.
\]

\end{theorem}

Following \cite[Theorem 3.2]{Preprint_PS}, there exists a sequence
$\{B_{n}\}_{n\in\mathbb{N}}$ such that, for every $m\in\mathbb{N}$,
\[%
\begin{array}
[c]{rcl}%
\displaystyle\frac{C_{\mathbb{C},{n}}}{C_{{\mathbb{C},{n-2}}}^{\frac{n-2}{n}}}
& = & B_{n}\\
&  &
\end{array}
\]
where
\[
B_{n}=\displaystyle\frac{2^{(n+2)/(2n)}}{\pi^{1/n}}\cdot\left(  \frac
{1}{A_{\frac{2n-4}{n-1}}^{2}}\right)  ^{\frac{n-2}{n}}.
\]

Now, making some algebraic manipulations, and keeping in mind that
\[
A_{p} = \sqrt{2} \left(  \frac{\Gamma((p+1)/2)}{\sqrt{\pi}}\right)  ^{1/p},
\]
we have
\[%
\begin{array}
[c]{rcl}%
B_{n} & = & \displaystyle \left(  \frac{\pi}{2}\right)  ^{(n-3)/(2n)}%
\cdot2^{3/(2n)} \cdot\frac{1}{\Gamma\left(  \frac{3n-5}{2n-2}\right)
^{(n-1)/n}}\\
&  &
\end{array}
\]
Now, from the continuity of the Gamma function, together with considering
equivalent infinities, one has that
\[%
\begin{array}
[c]{rcl}%
\displaystyle \lim_{n \rightarrow\infty} B_{n} & = & \displaystyle \lim_{n
\rightarrow\infty} \frac{ \left(  \frac{\pi}{2}\right)  ^{(n-3)/(2n)}
\cdot2^{3/(2n)}}{\Gamma\left(  \frac{3n-5}{2n-2}\right)  ^{(n-1)/n}} =
\displaystyle \frac{\left(  \frac{\pi}{2}\right)  ^{1/2} \cdot2^{0}}%
{\Gamma(3/2)} = \sqrt{2},
\end{array}
\]
using the known fact that $\Gamma(3/2) = \frac{\sqrt{\pi}}{2}$. Thus, we have
shown that
\[
\displaystyle \lim_{n \rightarrow\infty} \left( \displaystyle \frac
{C_{\mathbb{C},{n}}}{C_{{\mathbb{C},{n-2}}}^{\frac{n-2}{n}}}\right)  =
\sqrt{2}.
\]

Our aim now shall be to find a closed expression for the values of
$C_{\mathbb{C},{m}}$ in order to be able to study the asymptotic behavior.

\begin{theorem}
\label{complexrn} If $n$ is an even positive integer, then
\[
C_{\mathbb{C},n}=2^{\frac{n+2}{8}} \cdot\frac{\sqrt{2}}{\sqrt{\pi}} \cdot
r_{n}.
\]

\end{theorem}

\noindent The proof is very similar to that of Theorem \ref{realrn} and we
spare the details of it to the interested reader. The case $n$ odd has a very
similar formula. Next, the following result, of identical proof as in Theorem
\ref{asymptotic_real}, is now due:

\begin{theorem}
\label{asymptotic_complex} For the complex case,
\[
\displaystyle\lim_{n\rightarrow\infty}\frac{C_{\mathbb{C},n}}{C_{\mathbb{C}%
,n-2}}=2^{\frac{1}{4}}.\lim_{n\rightarrow\infty}\frac{r_{n}}{r_{n-2}}.
\]
In particular, assuming the convergence of $\{r_{n}\}$,
\[
\displaystyle\lim_{n\rightarrow\infty}\frac{C_{\mathbb{C},n}}{C_{\mathbb{C}%
,n-1}}=2^{\frac{1}{8}}.
\]

\end{theorem}

\subsection{Some remarks}

The following result was also obtained in \cite{Preprint_PS} as a consequence
of results from \cite{defant}, providing smaller constants for the complex case:

\begin{theorem}
\cite[Theorem 3.2]{Preprint_PS} For every positive integer $n$ and every
complex Banach spaces $X_{1},\ldots,X_{n},$
\[
\Pi_{(\frac{2n}{n+1};1)}(X_{1},\ldots,X_{n};\mathbb{C})=\mathcal{L}%
(X_{1},\ldots,X_{n};\mathbb{C})\text{ and }\left\Vert .\right\Vert _{\pi
(\frac{2n}{n+1};1)}\leq C_{\mathbb{C},n}\left\Vert .\right\Vert
\]
with
\begin{align*}
C_{\mathbb{C},n}  &  =\left(  \frac{2}{\sqrt{\pi}}\right) ^{n-1}\text{ for
}m=2,3,4,5,6,7,\\
C_{\mathbb{C},n}  &  \leq\frac{2^{\frac{n+2}{2n}}}{\pi^{1/n}}\left(  \frac
{1}{A_{\frac{2n-4}{n-1}}^{2}}\right)  ^{\frac{n-2}{n}}\left(  C_{\mathbb{C}%
,n-2}\right)  ^{\frac{n-2}{n}}\text{ for }n\geq8.
\end{align*}
In particular, for $8\leq n\leq14$ we have%
\[
C_{\mathbb{C},n}\leq\left(  \frac{1}{\pi^{1/n}}\right)  2^{\frac{n+4}{2n}%
}\left(  C_{\mathbb{C},n-2}\right)  ^{\frac{n-2}{n}}.
\]

\end{theorem}

By using the constants above we can obtain a closed formula with smaller
constants for the complex case. But, these new constants have the same
asymptotic behavior of the previous. It can be checked that
\[
C_{\mathbb{C},{14}}=\frac{2^{30/7}}{\pi^{19/14}}.
\]
In a similar fashion as the calculations we made for the real case in Theorem
2.3, it can be seen (we spare the details to the reader) that, for even values
of $n\geq16$:
\begin{equation}
C_{\mathbb{C},{n}}=\frac{2^{\frac{n}{8}+\frac{67}{n}+\frac{3}{4}}}{\pi
^{\frac{36}{n}+\frac{1}{2}}}\cdot\left(  \prod_{k=1}^{6}\Gamma\left(
\frac{6k+1}{4k+2}\right)  ^{2k+1}\right)  ^{1/n}r_{n}. \label{complex_even}%
\end{equation}
Evaluating $\prod_{k=1}^{6}\Gamma\left(  \frac{6k+1}{4k+2}\right)  ^{2k+1}$ we
obtain
\[
C_{\mathbb{C},{n}}=\frac{2^{\frac{n}{8}+\frac{67}{n}+\frac{3}{4}}}{\pi
^{\frac{36}{n}+\frac{1}{2}}}\cdot\sqrt[n]{0.003929571803} \cdot r_{n}%
\]

Of course, a similar procedure can be made for odd values of $n$.
\bigskip

\textbf{Acknowledgement.} This note was written while the second named author
was visiting the Facultad de Ciencias Matem\'{a}ticas de la Universidad
Complutense de Madrid. He thanks Prof. Seoane-Sep\'{u}lveda and the members of
the Facultad for their kind hospitality.


\end{document}